\documentclass[11pt]{article}
\usepackage[margin=1.1in]{geometry}
\usepackage{amsmath,amssymb,amsthm}
\usepackage{booktabs}
\usepackage{array}
\usepackage{tikz}
\usetikzlibrary{arrows.meta}
\usepackage[colorlinks=true,linkcolor=blue,citecolor=blue,urlcolor=blue]{hyperref}

\newtheorem{theorem}{Theorem}[section]
\newtheorem{lemma}[theorem]{Lemma}
\newtheorem{proposition}[theorem]{Proposition}
\newtheorem{corollary}[theorem]{Corollary}
\theoremstyle{definition}
\newtheorem{example}[theorem]{Example}
\newtheorem{remark}[theorem]{Remark}
\newtheorem{question}[theorem]{Question}

\newcommand{\LR}{\mathrm{LR}}
\newcommand{\Z}{\mathbb{Z}}
\newcommand{\N}{\mathbb{N}}
\newcommand{\norm}[1]{\left\lVert #1 \right\rVert}

\title{Tight instances of the Lonely Runner Conjecture:\\
one-entry modifications and the growth of speeds}
\author{Yuhan Zhang}
\date{}

\begin{document}
\maketitle

\begin{center}
Yangzhou University, Yangzhou, China\\
\texttt{yuhanzh6@gmail.com}
\end{center}

\begin{abstract}
For a set $V$ of $n-1$ distinct positive integers let
$\LR(V)=\max_{t}\min_{v\in V}\norm{vt}$, where $\norm{x}$ is the distance
from $x$ to the nearest integer. The Lonely Runner Conjecture asserts
$\LR(V)\ge 1/n$, and $V$ is called \emph{tight} if equality holds.
Goddyn and Wong determined which of the sets obtained from
$[n-1]=\{1,\dots,n-1\}$ by multiplying a single entry by an integer are
tight. We complete the classification of arbitrary one-entry
modifications $([n-1]\setminus\{r\})\cup\{w\}$, $w\ge n$: for $n\ge4$
such a set is tight if and only if it is of Goddyn--Wong type or is one
of the two sporadic sets $\{1,3,4,7\}$ and $\{1,3,4,5,9\}$ found by
Wills. The proof is elementary and rests on an exact description of the
region left uncovered by the retained speeds. We then show that the
largest speed $M(n)$ occurring in a tight one-entry modification
satisfies $M(n)<0.60\,n\log n+52\,n$ and
$\limsup_{n\to\infty}M(n)/(n\log n)=\tfrac12$; in particular, tight
sets are not confined to speeds $O(n)$. Along the way we give a short
proof of Liu's closed formula for $\LR([n-1]\setminus\{2\})$ and show
that $\LR([n-1]\setminus\{r\})>1/(n-1)$ for every $r\le n-2-2\sqrt n$,
extending the range $r\le(n-1)/2$ covered by Liu's theorem. For
two-entry modifications we show that no tight instance removes a speed
$q$ with $2\le q\le n/10$ once $n\ge\max(40,10q)$, whatever the second
removed speed and the inserted speeds are; an exact-arithmetic census
for $8\le n\le140$ with inserted speeds at most $8n$ finds only two
tight two-entry modifications, both already listed by Goddyn and Wong.
\end{abstract}

%======================================================================
\section{Introduction}
%======================================================================

\subsection{Background}

Let $\norm{x}$ denote the distance from $x\in\mathbb R$ to the nearest
integer. For a finite set $V\subset\N$ of \emph{speeds} put
\[
\LR(V)\;=\;\max_{t\in[0,1)}\ \min_{v\in V}\norm{vt};
\]
the maximum exists because $t\mapsto\min_{v}\norm{vt}$ is continuous on
the circle $\mathbb R/\Z$. The \emph{Lonely Runner Conjecture} of Wills
\cite{Wills1967} and Cusick \cite{Cusick1973} asserts that
$\LR(V)\ge\frac1n$ whenever $\lvert V\rvert=n-1$. In the usual picture,
$n$ runners with distinct constant speeds run around a circular track of
length one, and $V$ records the speeds relative to one distinguished
runner; we keep this convention throughout, so that $n$ always denotes
the number of runners and $V$ has $n-1$ elements.

The conjecture is known for $n\le7$
\cite{BetkeWills1972,CusickPomerance1984,BieniaEtAl1998,BohmanHolzmanKleitman2001,BarajasSerra2008}
and, in published form, for $n\in\{9,10\}$ \cite{Trakulthongchai2026};
the cases $n=8$ \cite{Rosenfeld2025eight}, $n\in\{11,12,13\}$
\cite{SungkawichaiTrakulthongchai2026} and $n\in\{14,15\}$
\cite{Allikvere2026} have been announced in computer-assisted preprints
(see also \cite{Rosenfeld2025nine} for an independent treatment of
$n=9$). We refer to the survey of Perarnau and Serra
\cite{PerarnauSerra2025} for the history of the problem. For general
$n$, Bedert \cite{Bedert2025} recently showed that $k$ speeds always
admit loneliness at least $\frac1{2k}+\frac1{k^{5/3+o(1)}}$, a
polynomial improvement over Tao's bound
$\frac1{2k}+\frac{(\log k)^{1-o(1)}}{k^2}$ \cite{Tao2018}.

This paper does not address the conjecture itself. It is concerned with
the structure of its \emph{equality case}.

\subsection{Tight instances}

By Dirichlet's approximation theorem, $\LR([n-1])=\frac1n$ for every
$n$, so the conjectured bound is best possible. A set $V$ of $n-1$
speeds with $\LR(V)=\frac1n$ is called \emph{tight}; dilations $kV$ of
tight sets are tight. For at most five speeds the tight sets are known
completely: up to dilation they are the baseline $[n-1]$ and, for four
and five speeds respectively, the sets
\[
T_1=\{1,3,4,7\}\quad(n=5),\qquad T_2=\{1,3,4,5,9\}\quad(n=6),
\]
found by Wills \cite{Wills1968} after a comment of P.~Flor; see
\cite{Cusick1973,Cusick1974} (at most three speeds), \cite{Chen1991}
(four speeds) and \cite[Thm.~3]{BohmanHolzmanKleitman2001} (five
speeds), and the account in \cite[\S4]{PerarnauSerra2025}. Wills also
found the tight set
$\{1,2,3,4,5,7,12\}$ with $n=8$. Goddyn and Wong \cite{GoddynWong2006}
recognised the latter as the first member of an infinite family
obtained by \emph{accelerating} one runner of the baseline, and listed
further tight sets such as $\{1,4,5,6,7,11,13\}$. In the open-problem
list of their survey, Perarnau and Serra state as Problem~1 the
characterisation of tight instances, noting that it ``has seen no
further progress since the work of Goddyn and Wong''
\cite[\S 11, Problem~1]{PerarnauSerra2025}.

Related work concerns sets that are \emph{almost} tight. Kravitz
\cite{Kravitz2021} conjectured that the loneliness of $n-1$ speeds is
either at least $\frac1{n-1}$ or of the form
$\frac{\sigma}{(n-1)\sigma+1}$ with $\sigma\in\N$; Fan and Sun
\cite{FanSun2026} disproved this for four speeds and proposed an
amended version. Giri and Kravitz \cite{GiriKravitz2026} showed that the
accumulation points of the $m$-th lonely runner spectrum form precisely
the $(m-1)$-st spectrum. Beck and Everett \cite{BeckEverett2026} proved
that every tight instance (and every counterexample) with $k$ speeds
satisfies a nontrivial integer linear relation of $\ell^1$-norm at most
$\min\bigl(2k+3,\frac{k+1}{k-1}\,\mathrm{flt}(k)\bigr)$, where
$\mathrm{flt}(k)$ is the flatness constant. In the opposite direction,
Bohman and Peng \cite{BohmanPeng2022} and Pomerance \cite{Pomerance2022}
showed that a set of $k$ speeds whose largest element lies in
$(k,\,2k-c\log^2k)$ is never tight, while tight sets with largest speed
$2k-\Theta(\log k)$ arise from the Goddyn--Wong criterion and a
construction of Erdős for the Jacobsthal function
\cite[\S4]{PerarnauSerra2025}.

\subsection{One-entry modifications}

For $r\in[n-1]$ and an integer $w$ write
\[
[n-1]_{r\to w}\;=\;\bigl([n-1]\setminus\{r\}\bigr)\cup\{w\},
\]
and $[n-1]_{r\mapsto mr}$ in the special case $w=mr$, $m\in\N$. Since
$[n-1]_{r\to w}$ has $n-1$ elements only if $w\notin[n-1]\setminus\{r\}$,
and $w=r$ returns the baseline, we always assume $w\ge n$.

\begin{theorem}[Goddyn--Wong {\cite[Thm.~2.3]{GoddynWong2006}}]\label{thm:GW}
Let $n\ge2$, $r\in[n-1]$ and $m\ge1$, and put $s=n-r$. Then
$\LR([n-1]_{r\mapsto mr})\ge\frac1n$, with equality if and only if
$n=2$, or $(n,r,m)=(3,1,4)$, or
\begin{equation}\label{eq:gwcond}
\gcd(r,b)>1\qquad\text{for every integer } b \text{ with } s\le b\le ms-1 .
\end{equation}
\end{theorem}

Goddyn and Wong also proved (Theorem~2.4) that, for fixed $r$, only
finitely many replacements $r\to w$ with $r\nmid w$ can be tight, and
remarked on page~7 that this ``partially explains why there are no
further sporadic speed sets resembling $T_1$ and $T_2$ where the speed
$2$ is replaced by an odd integer'' \cite[p.~7]{GoddynWong2006}.
Our first result settles the one-entry case completely.

\begin{theorem}[Classification of one-entry modifications]\label{thm:A}
Let $n\ge4$, $r\in[n-1]$ and $w\ge n$, and put $s=n-r$. Then
$[n-1]_{r\to w}$ is tight if and only if one of the following holds:
\begin{itemize}
\item[\textup{(i)}] $r\mid w$ and \eqref{eq:gwcond} holds with $m=w/r$;
\item[\textup{(ii)}] $(n,r,w)\in\{(5,2,7),\,(6,2,9)\}$, i.e.\ the set is
$T_1$ or $T_2$.
\end{itemize}
\end{theorem}

\begin{remark}\label{rem:A}
(a) Condition \textup{(i)} forces $2r>n-1$. Indeed, if $2r\le n-1$ then
$s\ge r+1$, and since $m\ge2$ the interval $[s,ms-1]$ contains at least
$s\ge r$ consecutive integers, hence an integer $b\equiv1\pmod r$, for
which $\gcd(r,b)=1$. Consequently \emph{no tight one-entry
modification removes a speed $r$ with $3\le r\le(n-1)/2$}, and the
speed $2$ is the only removable speed in this range, through $T_1$ and
$T_2$.

(b) For $n\le3$ the situation is elementary. For $n=2$ every set
$\{w\}$ is tight. For $n=3$ and $r=1$ Theorem~\ref{thm:GW} applies, and
the only tight example is $(n,r,w)=(3,1,4)$. For $n=3$ and $r=2$, an
even $w$ is covered by Theorem~\ref{thm:GW} (and gives no tight set),
while for odd $w$ all speeds of $\{1,w\}$ are odd, so $t=\frac12$ gives
$\LR(\{1,w\})=\frac12>\frac13$.
\end{remark}

The proof of Theorem~\ref{thm:A} is based on the \emph{uncovered region}
\[
U(n,r)=\Bigl\{t\in[0,1):\ \norm{it}>\tfrac1n\ \text{for all }
i\in[n-1]\setminus\{r\}\Bigr\},
\]
which the inserted speed $w$ alone must cover if $[n-1]_{r\to w}$ is
tight. We determine $U(n,r)$ exactly (Propositions~\ref{prop:comp}
and~\ref{prop:complarge}). For $2r>n-1$ this forces $r\mid w$, and
Theorem~\ref{thm:GW} applies. For $3\le r\le(n-1)/2$ the longest
component of $U(n,r)$ limits $w$, and a short argument with the
Jacobsthal function shows that the two components next to $\frac1r$
then impose incompatible conditions on $w\bmod r$. For $r=2$ the two
components of $U(n,2)$ around $\frac12$ leave exactly four candidates;
two of them are $T_1$ and $T_2$, and the other two are excluded by a
single explicit point.

Our second result concerns the size of the inserted speed.

\begin{theorem}[Extremal speed]\label{thm:B}
For $n\ge3$ let $M(n)$ be the largest $w$ such that $[n-1]_{r\to w}$ is
tight for some $r\in[n-1]$ and $w\ge n$ (with $M(n)=-\infty$ if there
is none). Then
\[
M(n)\;<\;0.60\,n\log n+52\,n\quad\text{for all }n\ge3,
\qquad\text{and}\qquad
\limsup_{n\to\infty}\frac{M(n)}{n\log n}\;=\;\frac12 .
\]
The limit superior is attained along $n=p\#+2$, where $p\#$ is the
product of the primes $\le p$, by the sets $[n-1]_{r\mapsto mr}$ with
$r=p\#$ and $m=\frac{q-1}{2}$, $q$ the prime following $p$.
\end{theorem}

\begin{corollary}\label{cor:B}
There is no constant $C$ such that every tight set with $n$ runners has
all speeds at most $Cn$.
\end{corollary}

Goddyn and Wong observed that, for a tight set $[n-1]_{r\mapsto mr}$,
the speed $r$ ``(and thus $n$) must be quite large compared to integers
$m$ and $s$'', and they studied the least $r$ admissible for given
$(m,s)$ \cite[Sec.~4]{GoddynWong2006}.
Theorem~\ref{thm:B} answers the inverse question: given the number of
runners, how large can the accelerated speed be?

\subsection{Further results and organisation}

Section~\ref{sec:prelim} describes the uncovered region. Its last part
deals with the \emph{deletion values} $\LR([n-1]\setminus\{r\})$, which
bound $\LR([n-1]_{r\to w})$ from above. For $r\le\frac{n-1}2$ the exact
values are due to Liu: the case $s=1$ of his theorem
\cite[Thm.~18]{Liu2008} gives
$\LR([n-1]\setminus\{r\})=2/D_r$ with $D_r$ the least integer $\ge n+r$
coprime to $r$; in particular
\[
\LR\bigl([n-1]\setminus\{2\}\bigr)=\frac{2}{2\lfloor n/2\rfloor+3}
\qquad(n\ge5),
\]
for which we include a short self-contained proof
(Theorem~\ref{thm:r2closed}). Beyond Liu's range, a single rational
witness gives $\LR([n-1]\setminus\{r\})>\frac1{n-1}$ whenever some
integer in $[n+r,2n-3]$ is coprime to $r$, and in particular for all
$\frac{n-1}2<r\le n-2-2\sqrt n$ (Theorem~\ref{thm:delwit} and
Corollary~\ref{cor:delwit}). Sections~\ref{sec:thmA}
and~\ref{sec:thmB} prove Theorems~\ref{thm:A} and~\ref{thm:B}.

Section~\ref{sec:multi} recalls the multi-acceleration theorem of
Goddyn and Wong, which deserves to be better known, and extends the
growth bound to the families it produces. Section~\ref{sec:twoswap}
treats two-entry modifications. We show by a covering-density argument
that a two-entry modification removing a speed $q$ with
$2\le q\le n/10$ is never tight once $n\ge\max(40,10q)$, irrespective of
the other removed speed and of the inserted speeds
(Theorem~\ref{thm:qsmall}). An exact census covers $8\le n\le140$ with
inserted speeds at most $8n$ (Theorem~\ref{thm:twoswap}).
Section~\ref{sec:comp} describes the computations,
Section~\ref{sec:open} collects open problems, and
Appendix~\ref{sec:shifted} comments on the shifted variant of the
conjecture.

\paragraph{On the methods.} The proofs of Theorems~\ref{thm:A}
and~\ref{thm:B} use only Dirichlet's theorem, interval-covering
arguments, Kanold's bound on the Jacobsthal function, and classical
estimates for Chebyshev's function. Beyond these, they rely on
Theorem~\ref{thm:GW} and on the tightness of $T_1$ and $T_2$
\cite{Wills1968}. The only finite verification they need is the
inequality $\theta(x)\ge0.84\,x$ for $101\le x<523$, which can be read
off from a table of primes. The census of Section~\ref{sec:twoswap}, the
small cases of Theorem~\ref{thm:r23} and the spectrum search of
Section~\ref{sec:spectrum} are computer-assisted: each is
a proof by finite enumeration in exact rational arithmetic, and the
programs and logs are publicly available.

%======================================================================
\section{The uncovered region}\label{sec:prelim}
%======================================================================

Throughout, $\log$ is the natural logarithm. For a finite set $W$ we
put $f_W(t)=\min_{v\in W}\norm{vt}$, and for $w\in\N$ we write
$B_w=\{t\in[0,1):\norm{wt}\le\frac1n\}$ for the \emph{bad set} of $w$.
For $R\subseteq[n-1]$ let
\[
U(n,R)=\Bigl\{t\in[0,1):\ \norm{it}>\tfrac1n\ \text{for all }
i\in[n-1]\setminus R\Bigr\},\qquad U(n,r)=U(n,\{r\}).
\]
Since enlarging $R$ removes constraints,
\begin{equation}\label{eq:monotone}
R\subseteq R'\quad\Longrightarrow\quad U(n,R)\subseteq U(n,R').
\end{equation}

\subsection{Basic lemmas}

\begin{lemma}[Dirichlet]\label{lem:dir}
For every $t\in\mathbb R$ there is $i\in[n-1]$ with
$\norm{it}\le\frac1n$. Consequently $U(n,\emptyset)=\emptyset$ and
$\LR([n-1])=\frac1n$.
\end{lemma}

\begin{proof}
The first assertion is Dirichlet's theorem. For the second, take
$t=a/n$ with $\gcd(a,n)=1$; then $\norm{it}\ge\frac1n$ for all
$i\in[n-1]$.
\end{proof}

\begin{lemma}[Coverage]\label{lem:cov}
Let $V=([n-1]\setminus R)\cup W$ with $\LR(V)\le\frac1n$. Then
$U(n,R)\subseteq\bigcup_{w\in W}B_w$.
\end{lemma}

\begin{proof}
For $t\in U(n,R)$ every retained speed $i$ has $\norm{it}>\frac1n$,
while $f_V(t)\le\frac1n$; hence $\norm{wt}\le\frac1n$ for some $w\in W$.
\end{proof}

\begin{lemma}[Bad intervals]\label{lem:int}
Let $n\ge3$ and $w\in\N$. Then $B_w$ is the disjoint union (on the
circle) of the closed intervals
$\bigl[\frac jw-\frac1{wn},\,\frac jw+\frac1{wn}\bigr]$, $j\in\Z$, each
of length $\frac{2}{wn}$. In particular, every connected subset of
$B_w$ lies in one of these intervals.
\end{lemma}

\begin{lemma}[Divisibility]\label{lem:div}
If $2r>n-1$ and $[n-1]_{r\to w}$ is tight, then $r\mid w$.
\end{lemma}

\begin{proof}
Since $2r>n-1$, no $i\in[n-1]\setminus\{r\}$ is divisible by $r$, so
$\norm{i/r}\ge\frac1r>\frac1n$ for these $i$, i.e.\
$\frac1r\in U(n,r)$. By Lemma~\ref{lem:cov},
$\norm{w/r}\le\frac1n<\frac1r$. As $\norm{w/r}\in\frac1r\Z$, this
forces $r\mid w$.
\end{proof}

\subsection{\texorpdfstring{The components of $U(n,r)$}{The components of U(n,r)}}

For $r\ge2$ and $x\not\equiv0\pmod r$ let
\[
F(x)=\max\{i\le n-1:\ i\ne r,\ i\equiv x\pmod r\},
\qquad
I(n,r)=\min\{i\ge n-r:\ \gcd(i,r)=1\}.
\]
The block $\{s,s+1,\dots,n-1\}$, $s=n-r$, consists of $r$ consecutive
integers, so every residue class modulo $r$ has exactly one
representative in it. The only multiple of $r$ that is excluded from
$F$, namely $r$ itself, lies in the zero class. Hence, for every unit
$x$ modulo $r$, $F(x)$ is the representative of $x$ in
$\{s,\dots,n-1\}$, and
\begin{equation}\label{eq:Fset}
\{F(x):\ \gcd(x,r)=1\}=\{i\in[s,n-1]:\ \gcd(i,r)=1\},\qquad
I(n,r)=\min_{\gcd(x,r)=1}F(x).
\end{equation}
For a reduced residue $p$ modulo $r$ we write $u_p\equiv p^{-1}\pmod r$.

\begin{proposition}[Components, middle range]\label{prop:comp}
Let $2\le r\le\frac{n-1}2$, so that $s=n-r\ge r+1$. Then
\[
U(n,r)=\bigsqcup_{\substack{0\le p<r\\ \gcd(p,r)=1}}\bigl(L_p\sqcup R_p\bigr),
\qquad
L_p=\Bigl(\tfrac pr-\tfrac{s}{rnF(u_p)},\ \tfrac pr-\tfrac1{2rn}\Bigr),\quad
R_p=\Bigl(\tfrac pr+\tfrac1{2rn},\ \tfrac pr+\tfrac{s}{rnF(r-u_p)}\Bigr),
\]
and all these intervals are nonempty.
\end{proposition}

\begin{proof}
\emph{Trapping.} Let $t\in U(n,r)$. By Lemma~\ref{lem:dir} some
$i\in[n-1]$ has $\norm{it}\le\frac1n$; by definition of $U(n,r)$ this
$i$ must be $r$, so $\lvert t-\frac jr\rvert\le\frac1{rn}$ for some
$j\in\Z$. If $g=\gcd(j,r)>1$, then $r'=r/g\in[n-1]\setminus\{r\}$
satisfies
$\norm{r't}\le r'\lvert t-\frac jr\rvert\le\frac1{gn}<\frac1n$, which is
impossible. Hence $j\equiv p\pmod r$ with $\gcd(p,r)=1$, and we may
write $t=\frac pr+\delta$ with $\lvert\delta\rvert\le\frac1{rn}$. It
remains to determine which such $\delta$ satisfy $\norm{it}>\frac1n$
for all $i\in[n-1]\setminus\{r\}$.

\emph{Multiples of $r$.} For $i=kr$ with $2\le k\le\frac{n-1}r$ we
have $\norm{it}=kr\lvert\delta\rvert$, because
$kr\lvert\delta\rvert\le\frac kn<\frac12$. The constraint reads
$\lvert\delta\rvert>\frac1{krn}$, and since $2r\le n-1$ the strongest
one is
\begin{equation}\label{eq:inner}
\lvert\delta\rvert>\tfrac1{2rn}.
\end{equation}

\emph{Non-multiples.} Let $r\nmid i$, $\rho\equiv ip\pmod r$ with
$1\le\rho\le r-1$, and $d=\min(\rho,r-\rho)$. Then
$it\equiv x:=\frac\rho r+i\delta\pmod1$, and
$\lvert i\delta\rvert\le\frac{n-1}{rn}<\frac1r$, so
$x\in(\frac{\rho-1}r,\frac{\rho+1}r)\subseteq[0,1]$. If $i\delta$ moves
$x$ towards the integer nearest to $\frac\rho r$ (for $\rho=\frac r2$
both directions count), then $\norm{it}=\frac dr-\lvert i\delta\rvert$.
Otherwise $\rho\ne\frac r2$, so $r\ge3$ and $d\le\frac{r-1}2$, and
$\norm{it}\ge\min\bigl(\frac dr,\frac{r-d-1}r\bigr)\ge\frac1r>\frac1n$.
So the only active constraint is
$\lvert\delta\rvert<\frac{dn-r}{rni}$, on the side that moves $x$
towards an integer.

If $d\ge2$ this constraint is void on the trap, because
$\frac{dn-r}{rni}\ge\frac{2n-r}{rn(n-1)}=\frac{n+s}{rn(n-1)}>\frac1{rn}$.
If $d=1$, then $i\equiv\pm u_p\pmod r$. The class $i\equiv u_p$ (so
$\rho=1$) constrains $\delta<0$, the class $i\equiv-u_p$ (so
$\rho=r-1$) constrains $\delta>0$, and in each class the largest
admissible $i$, namely $F(u_p)$ resp.\ $F(r-u_p)$, gives the strongest
bound $\lvert\delta\rvert<\frac{s}{rnF(\cdot)}$. (For $r=2$ the two
classes coincide, and $F(1)$ bounds both sides.)

Together with \eqref{eq:inner} this yields exactly $L_p\sqcup R_p$.
Both intervals are nonempty because $F\le n-1<2s$. Conversely, every
point of $L_p\sqcup R_p$ satisfies all the constraints above.
\end{proof}

\begin{proposition}[Components, large range]\label{prop:complarge}
Let $2\le r\le n-1$ with $2r>n-1$, i.e.\ $r\ge s$. Then
\[
U(n,r)=\bigsqcup_{\substack{0\le p<r\\ \gcd(p,r)=1}}
\Bigl(\tfrac pr-\tfrac{s}{rnF(u_p)},\ \tfrac pr+\tfrac{s}{rnF(r-u_p)}\Bigr),
\]
where in this range $F(x)=x$ for $s\le x<r$ and $F(x)=x+r$ for
$1\le x<s$. Moreover $\sup_{t\in U(n,r)}\norm{rt}=\frac{s}{n\,I(n,r)}$.
\end{proposition}

\begin{proof}
The trapping step and the analysis of non-multiples are exactly as in
Proposition~\ref{prop:comp}; they used only $F(x)\ge s$, which holds by
\eqref{eq:Fset}. Since $2r>n-1$ there are no multiples $kr\le n-1$ with
$k\ge2$, so \eqref{eq:inner} disappears and $\frac pr\in U(n,r)$. On the
component around $\frac pr$ we have
$\norm{rt}=r\lvert t-\frac pr\rvert$, whose supremum is
$\frac sn\min\bigl(F(u_p),F(r-u_p)\bigr)^{-1}$; maximising over $p$ and
using \eqref{eq:Fset} gives $\frac{s}{nI(n,r)}$.
\end{proof}

\begin{corollary}[Longest component]\label{cor:length}
Let $\ell(n,r)$ be the length of the longest connected component of
$U(n,r)$. Then
\[
\ell(n,r)=\frac{2s-I(n,r)}{2rn\,I(n,r)}\ \ \text{if } 2r\le n-1,
\qquad
\ell(n,r)=\frac{s}{rn}\max_{\gcd(u,r)=1}
\Bigl(\frac1{F(u)}+\frac1{F(r-u)}\Bigr)\ \ \text{if } 2r>n-1 .
\]
\end{corollary}

\begin{proof}
In the middle range $\lvert R_p\rvert=\frac{2s-F(r-u_p)}{2rnF(r-u_p)}$
and similarly for $L_p$; these lengths decrease in $F$, so by
\eqref{eq:Fset} the maximum is attained at $F=I(n,r)$. The large range
is immediate from Proposition~\ref{prop:complarge}.
\end{proof}

For $r=2$ the uncovered region is a pair of intervals around
$\frac12$, which we call the \emph{shell}; it will be used both in
Section~\ref{sec:thmA} and in Section~\ref{sec:twoswap}.

\begin{corollary}[The shell]\label{cor:U2}
Let $n\ge5$. Then $U(n,2)=L_-\sqcup L_+$ with
\[
L_-=\bigl(\tfrac12-\beta,\ \tfrac12-\tfrac1{4n}\bigr),\qquad
L_+=\bigl(\tfrac12+\tfrac1{4n},\ \tfrac12+\beta\bigr),\qquad
\beta=\begin{cases}\frac1{2n}& n\text{ odd},\\[2pt]
\frac{n-2}{2n(n-1)}& n\text{ even},\end{cases}
\]
and both components have length
$\Delta(n)=\beta-\frac1{4n}$, which equals $\frac1{4n}$ for odd $n$
and $\frac{n-3}{4n(n-1)}$ for even $n$.
\end{corollary}

\begin{proof}
Apply Proposition~\ref{prop:comp} with $r=2$ (the middle range
$4\le n-1$): the only reduced centre is $\frac12$, and $F(1)$ is the
largest odd integer $\le n-1$.
\end{proof}

\begin{figure}[t]
\centering
\begin{tikzpicture}[>=Stealth]
  \draw[->] (0.6,0) -- (9.1,0) node[below left] {$t$};
  \fill[blue!12] (1.3,-0.10) rectangle (3.9,0.10);
  \fill[blue!12] (5.1,-0.10) rectangle (7.7,0.10);
  \node at (2.6,0.32) {$L_p$};
  \node at (6.4,0.32) {$R_p$};
  \foreach \x in {1.3,3.9,5.1,7.7}{\draw (\x,-0.16) -- (\x,0.16);}
  \fill (4.5,0) circle (0.03);
  \node[below] at (4.5,-0.10) {$\frac pr$};
  \draw[<-,thick] (1.3,0.10) -- (1.3,0.85);
  \node[above] at (1.3,0.85) {\scriptsize speed $F(u_p)$};
  \draw[<-,thick] (3.9,0.10) -- (3.9,0.85);
  \node[above] at (3.55,0.85) {\scriptsize speed $2r$};
  \draw[<-,thick] (5.1,0.10) -- (5.1,0.85);
  \node[above] at (5.5,0.85) {\scriptsize speed $2r$};
  \draw[<-,thick] (7.7,0.10) -- (7.7,0.85);
  \node[above] at (7.7,0.85) {\scriptsize speed $F(r-u_p)$};
  \fill[red!25] (5.25,-0.95) rectangle (7.55,-0.70);
  \node at (4.85,-1.15) [right] {\scriptsize a bad interval
  $[\frac jw-\frac1{wn},\frac jw+\frac1{wn}]$ of $w$};
\end{tikzpicture}
\caption{The two components of $U(n,r)$ next to $p/r$ in the middle
range $2r\le n-1$ (Proposition~\ref{prop:comp}). The inner endpoints are
determined by the speed $2r$, the outer ones by $F(u_p)$ and
$F(r-u_p)$. If $[n-1]_{r\to w}$ is tight, each component lies in a
single bad interval of $w$.}
\label{fig:components}
\end{figure}

\subsection{Deletion values}

Adding a speed cannot increase $\LR$, and $[n-1]\setminus\{r\}$ is
contained in $[n-1]$. Therefore
\begin{equation}\label{eq:sandwich}
\LR\bigl([n-1]_{r\to w}\bigr)\ \le\ \LR\bigl([n-1]\setminus\{r\}\bigr)
\qquad\text{and}\qquad
\LR\bigl([n-1]\setminus\{r\}\bigr)\ \ge\ \tfrac1n .
\end{equation}
The Lonely Runner Conjecture for the $n-2$ speeds
$[n-1]\setminus\{r\}$ predicts the stronger bound
$\LR([n-1]\setminus\{r\})\ge\frac1{n-1}$, with equality for $r=n-1$ by
Lemma~\ref{lem:dir}. For $r\le\frac{n-1}2$ the deletion values are in
fact known exactly: Liu \cite[Thm.~18]{Liu2008} determined
$\kappa([m]\setminus\{k,2k,\dots,sk\})$ for $m\ge(s+1)k$, and the case
$s=1$ (with $m=n-1$, $k=r$) gives
$\LR([n-1]\setminus\{r\})=2/D_r$, where $D_r$ is the least integer
$D\ge n+r$ with $\gcd(D,r)=1$. The following witness argument recovers
the lower bound $2/D_r$ and extends it beyond $r\le\frac{n-1}2$.

\begin{theorem}[Rational witnesses]\label{thm:delwit}
Let $n\ge3$, $r\in[n-1]$, and let $D\ge n+r$ be an integer with
$\gcd(D,r)=1$. Then
\[
\LR\bigl([n-1]\setminus\{r\}\bigr)\ \ge\ \frac2D .
\]
Consequently $\LR([n-1]\setminus\{r\})>\frac1{n-1}$ if some integer in
$[n+r,\,2n-3]$ is coprime to $r$, and $\LR([n-1]\setminus\{r\})\ge
\frac1{n-1}$ if some integer in $[n+r,\,2n-2]$ is coprime to $r$.
\end{theorem}

\begin{proof}
Choose $a$ with $ar\equiv1\pmod D$ and put $t=a/D$; then
$\gcd(a,D)=1$. Let $i\in[n-1]\setminus\{r\}$. Since $0<i<D$, we have
$ia\not\equiv0\pmod D$. If $ia\equiv\pm1\pmod D$, then
$i\equiv\pm r\pmod D$, so $i=r$ or $i\ge D-r\ge n$; both are excluded.
Hence $ia\bmod D\in[2,D-2]$, i.e.\ $\norm{it}\ge\frac2D$. The
consequences follow because $\frac2D>\frac1{n-1}$ for $D\le2n-3$ and
$\frac2D\ge\frac1{n-1}$ for $D\le2n-2$.
\end{proof}

Recall that the \emph{Jacobsthal function} $j(r)$ is the least $k$ such
that every $k$ consecutive integers contain one coprime to $r$. Kanold
proved $j(r)\le2^{\omega(r)}$, where $\omega(r)$ is the number of
distinct prime factors of $r$ \cite{Kanold1967}.

\begin{corollary}\label{cor:delwit}
If $n-r-2\ge j(r)$, then $\LR([n-1]\setminus\{r\})>\frac1{n-1}$. In
particular this holds for every $r\le n-2-2\sqrt n$.
\end{corollary}

\begin{proof}
The interval $[n+r,2n-3]$ contains $n-r-2$ consecutive integers. For
the second claim note that
$j(r)\le2^{\omega(r)}\le\tau(r)\le2\sqrt r<2\sqrt n$, where $\tau(r)$ is
the number of divisors of $r$.
\end{proof}

For $r=2$ the witness of Theorem~\ref{thm:delwit} is optimal; the
result is the case $r=2$ of Liu's theorem \cite[Thm.~18]{Liu2008}, and we
include a short self-contained proof.

\begin{theorem}[Deletion value for $r=2$; Liu]\label{thm:r2closed}
For every $n\ge5$,
\[
\LR\bigl([n-1]\setminus\{2\}\bigr)=\frac{2}{D},\qquad
D=2\lfloor n/2\rfloor+3 .
\]
\end{theorem}

\begin{proof}
Put $V=[n-1]\setminus\{2\}$. The integer $D$ is odd and equals $n+2$ or
$n+3$, so Theorem~\ref{thm:delwit} gives $\LR(V)\ge\frac2D$.

For the upper bound let $t\in[0,1)$. Among the $n-1$ points
$0,t,\dots,(n-2)t$ of the circle two are at distance at most
$\frac1{n-1}$, so some $d\in[n-2]$ has
$\norm{dt}\le\frac1{n-1}\le\frac2D$; the last inequality is
$D\le2n-2$, valid for $n\ge5$. If $d\ne2$ we are done. Suppose
$d=2$, and write $2t=m+\epsilon$ with $m\in\Z$ and
$\lvert\epsilon\rvert=e=\norm{2t}\le\frac1{n-1}$.
\begin{itemize}
\item If $e\le\frac1D$, then $\norm{4t}=2e\le\frac2D$ (here $4\in V$,
and $2e\le\frac12$).
\item If $e>\frac1D$ and $m$ is even, then
$\norm{t}=\frac e2\le\frac1{2(n-1)}<\frac2D$.
\item If $e>\frac1D$ and $m$ is odd, let $i_0$ be the largest odd
element of $V$, namely $n-1$ or $n-2$; in both cases $i_0=D-4\ge3$.
Then $i_0t\equiv\frac12+\frac{i_0\epsilon}2\pmod1$ and
$\frac{i_0e}2\le\frac12$, so
\[
\norm{i_0t}=\frac12-\frac{i_0e}2<\frac12-\frac{D-4}{2D}=\frac2D .
\]
\end{itemize}
In every case $f_V(t)\le\frac2D$.
\end{proof}

\begin{remark}
The witness bound is not sharp in general: for $(n,r)=(5,3)$ it gives
$\frac14$, whereas $t=\frac13$ shows $\LR(\{1,2,4\})\ge\frac13$. Note
also that $\frac2D>\frac1{n-1}$ in Theorem~\ref{thm:r2closed}, so for
$r=2$ the interval in \eqref{eq:sandwich} contains the whole range
$(\frac1n,\frac1{n-1})$.
\end{remark}

%======================================================================
\section{Proof of Theorem~\ref{thm:A}}\label{sec:thmA}
%======================================================================

\subsection{\texorpdfstring{The middle range $3\le r\le\frac{n-1}2$}{The middle range 3 <= r <= (n-1)/2}}

Let $3\le r\le\frac{n-1}2$, $s=n-r\ge r+1$, $w\ge n$, and suppose that
$V=[n-1]_{r\to w}$ is tight. Write
\[
F_+=F(1),\qquad F_-=F(r-1),\qquad w\equiv c\pmod r,\quad
c\in\bigl(-\tfrac r2,\tfrac r2\bigr].
\]
By Lemma~\ref{lem:cov}, Proposition~\ref{prop:comp} (with $p=1$) and
continuity, $\norm{wt}\le\frac1n$ on the closures $\overline{L_1}$ and
$\overline{R_1}$ of the two components next to $t_0=\frac1r$.

\begin{lemma}[Windows]\label{lem:win}
Assume moreover $w<r(n-2)$. Then
\begin{itemize}
\item[\textup{(R)}] $\norm{wt}\le\frac1n$ on $\overline{R_1}$ implies
$c\le\frac rn-\frac{ws}{nF_-}$;
\item[\textup{(L)}] $\norm{wt}\le\frac1n$ on $\overline{L_1}$ implies
$c\ge\frac{ws}{nF_+}-\frac rn$.
\end{itemize}
\end{lemma}

\begin{proof}
We prove (R); (L) is symmetric. Let $\overline{R_1}=[a,b]$ with
$a=t_0+\frac1{2rn}$ and $b=t_0+\frac{s}{rnF_-}$. By
Lemma~\ref{lem:int} there is $j\in\Z$ with
$wb-\frac1n\le j\le wa+\frac1n$. Write $w=Kr+c$, so that
$wt_0=K+\frac cr$, and put $J=j-K$. Then
\[
\frac cr+\frac{ws}{rnF_-}-\frac1n\ \le\ J\ \le\
\frac cr+\frac{w}{2rn}+\frac1n .
\]
If $J\ge1$, the right inequality gives
$c\ge r-\frac w{2n}-\frac rn>\frac r2$, where the last step is
equivalent to $w<r(n-2)$; this contradicts $c\le\frac r2$. If
$J\le-1$, the left inequality gives
$c\le-r+\frac rn-\frac{ws}{nF_-}<-\frac r2$, again impossible. Hence
$J=0$, and the left inequality is the claim.
\end{proof}

\begin{lemma}[Signs]\label{lem:sign}
If $s\ge r+1$ and $w\ge n$, then $\frac rn-\frac{ws}{nF_-}<0$ and
$\frac{ws}{nF_+}-\frac rn>0$.
\end{lemma}

\begin{proof}
Both inequalities mean $ws>rF_\pm$. Since $F_\pm\le n-1$ and $w\ge n$,
it suffices that $ns>r(n-1)$, i.e.\ $n(s-r)+r>0$.
\end{proof}

\begin{lemma}[Key inequality]\label{lem:key}
Let $r\ge3$, $n\ge2r+1$, $s=n-r$ and $I=I(n,r)$. Then $I<2s$ and
\[
\frac{4rI}{2s-I}<r(n-2).
\]
\end{lemma}

\begin{proof}
Put $d=I-s\ge0$ and $f(n)=\frac{s(n-6)}{n+2}$, which is positive since
$n\ge7$. A short computation shows that the assertion is equivalent to
$d<f(n)$, which also gives $I<s+s\frac{n-6}{n+2}<2s$. The derivative of
$f$ (with $r$ fixed) has the sign of $n^2+4n-8r-12$, which is at least
$4r^2+4r-7>0$ for $n\ge2r+1$; so $f$ is increasing on $[2r+1,\infty)$.
The $d$ consecutive integers $s,\dots,s+d-1$ all have a common factor
with $r$, hence
\[
d\ \le\ j(r)-1\ \le\ 2^{\omega(r)}-1 .
\]

\emph{Case $\omega(r)\ge2$.} Then $r\ge6$, and $r$ is at least the
product of $\omega(r)$ distinct primes, so $r\ge2\cdot3^{\omega(r)-1}$.
Since $2^k-1\le3^{k-1}$ for all $k\ge1$, we get $d\le\frac r2$. The
inequality $\frac r2<f(n)$ is equivalent to
$2n^2-3(r+4)n+10r>0$. At $n=2r+1$ the left side equals
$2r^2-9r-10>0$ (as $r\ge6$), and its derivative $4n-3(r+4)$ is positive
for $n\ge2r+1$. Hence $d\le\frac r2<f(n)$.

\emph{Case $r=p^a$ a prime power.} Then $j(r)=2$, so $d\le1$, and
$d=0$ unless $p\mid s$. If $r\ge4$, then
$f(n)\ge f(2r+1)=\frac{(r+1)(2r-5)}{2r+3}>1$, because $2r^2-5r-8>0$. If
$r=3$ and $3\mid s$, then $s\ge6$, so $n\ge9$ and
$f(n)\ge f(9)=\frac{18}{11}>1$. In all cases $d<f(n)$.
\end{proof}

\begin{proposition}\label{prop:mid}
For $3\le r\le\frac{n-1}2$ and $w\ge n$, the set $[n-1]_{r\to w}$ is
not tight.
\end{proposition}

\begin{proof}
Suppose it is tight. By Lemma~\ref{lem:cov} the longest component of
$U(n,r)$ lies in one bad interval of $w$. Comparing lengths with
Corollary~\ref{cor:length} and Lemma~\ref{lem:int} gives
$\frac{2s-I}{2rnI}\le\frac2{wn}$, i.e.\ $w\le\frac{4rI}{2s-I}$, and
hence $w<r(n-2)$ by Lemma~\ref{lem:key}. Now Lemmas~\ref{lem:win}
and~\ref{lem:sign} give $c<0$ from (R) and $c>0$ from (L), a
contradiction.
\end{proof}

\subsection{\texorpdfstring{The boundary case $r=2$}{The boundary case r=2}}

\begin{proposition}\label{prop:r2}
Let $n\ge5$ and $w\ge n$. Then $[n-1]_{2\to w}$ is tight if and only if
$(n,w)\in\{(5,7),(6,9)\}$.
\end{proposition}

\begin{proof}
The sets $T_1$ and $T_2$ are tight by \cite{Wills1968}. Conversely, let
$V=[n-1]_{2\to w}$ be tight, and let $L_\pm$, $\beta$, $\Delta=\Delta(n)$
be as in Corollary~\ref{cor:U2}. By Lemmas~\ref{lem:cov}
and~\ref{lem:int}, $\overline{L_-}$ lies in a single bad interval
$[c-h,c+h]$ with $c=\frac jw$ and $h=\frac1{wn}$, that is,
\begin{equation}\label{eq:shellwin}
c-h\le\tfrac12-\beta,\qquad c+h\ge\tfrac12-\tfrac1{4n},\qquad
2h\ge\Delta .
\end{equation}

\emph{The centre lies below $\frac12$.} If $c\ge\frac12$, then
$c-h\ge\frac12-\frac1{wn}>\frac12-\frac1{4n}$ because $w>4$, so the
interval misses the right endpoint of $\overline{L_-}$.

\emph{Parity.} The largest centre $\frac jw$ below $\frac12$ is
$\frac12-\frac1w$ if $w$ is even and $\frac12-\frac1{2w}$ if $w$ is odd.
Inserting it into the middle condition of \eqref{eq:shellwin} gives
$w\ge4(n-1)$ for even $w$ and $w\ge2(n-2)$ for odd $w$.

\emph{Length.} The last condition of \eqref{eq:shellwin} reads
$w\le\frac2{n\Delta}$, i.e.\ $w\le8$ for odd $n$ and
$w\le\frac{8(n-1)}{n-3}$ for even $n$.

For even $w$ the bounds $4(n-1)\le w\le\frac2{n\Delta}$ force $n\le3$
for odd $n$ and $n\le5$ for even $n$, contradicting $n\ge5$. So $w$ is
odd and $2(n-2)\le w\le\frac2{n\Delta}$. For odd $n$ this gives $n=5$
and $w=7$. For even $n$ it gives $(n-2)(n-3)\le4(n-1)$, i.e.\
$n^2-9n+10\le0$, whence $n=6$ and $w\in\{9,11,13\}$.

It remains to exclude $(n,w)=(6,11)$ and $(6,13)$. The point
$t^*=\frac{87}{200}$ lies in
$L_-=\bigl(\frac{13}{30},\frac{11}{24}\bigr)$, and
\[
\norm{t^*}=\tfrac{87}{200},\ \ \norm{3t^*}=\tfrac{61}{200},\ \
\norm{4t^*}=\tfrac{52}{200},\ \ \norm{5t^*}=\tfrac{35}{200},\ \
\norm{11t^*}=\tfrac{43}{200},\ \ \norm{13t^*}=\tfrac{69}{200},
\]
all of which exceed $\frac16$. Hence $\LR(V)>\frac16$ for
$w\in\{11,13\}$.
\end{proof}

\begin{remark}\label{rem:r2boundary}
The window argument of Lemma~\ref{lem:win} requires $w<r(n-2)$, whereas
for $r=2$ the covering condition forces $w\ge2(n-2)=r(n-2)$. The two
arguments meet exactly at the boundary, which explains why $r=2$ is the
only removed speed in the middle range that admits tight modifications.
\end{remark}

\subsection{Completion of the proof}

\begin{proof}[Proof of Theorem~\ref{thm:A}]
\emph{$2r>n-1$.} By Lemma~\ref{lem:div}, $r\mid w$, and
Theorem~\ref{thm:GW} applies; the exceptions $n=2$ and $(3,1,4)$ do not
occur for $n\ge4$. This case contains $r=2$ for $n=4$.

\emph{$r=1$.} Here $w=1\cdot w$ is of Goddyn--Wong form, and
\eqref{eq:gwcond} fails because $\gcd(1,b)=1$; so the set is not tight,
in accordance with (i).

\emph{$r=2$, $n\ge5$.} This is Proposition~\ref{prop:r2}; condition (i)
fails by Remark~\ref{rem:A}(a).

\emph{$3\le r\le\frac{n-1}2$.} This is Proposition~\ref{prop:mid};
again (i) fails by Remark~\ref{rem:A}(a).
\end{proof}

%======================================================================
\section{Proof of Theorem~\ref{thm:B}}\label{sec:thmB}
%======================================================================

Let $\theta(x)=\sum_{p\le x}\log p$ denote Chebyshev's function. By
Theorem~\ref{thm:A}, for $n\ge4$ a tight set $[n-1]_{r\to w}$ is either
$T_1$ or $T_2$ (so $w\le9$), or $w=mr$ with $m\ge2$ and
\eqref{eq:gwcond}. In the latter case $s\ge2$: for $s=1$ the integer
$b=1$ would violate \eqref{eq:gwcond}. Moreover every prime $b$ in
$[s,ms-1]$ satisfies $\gcd(r,b)>1$, i.e.\ $b\mid r$. As these primes are
distinct, their product is at most $r<n$, so
\begin{equation}\label{eq:theta}
\theta(ms-1)-\theta(s-1)\ \le\ \log r\ <\ \log n .
\end{equation}
For $n=3$ the only tight one-entry modification is $[2]_{1\mapsto4}$,
so $M(3)=4$.

\begin{proof}[Proof of the upper bound]
We use the classical estimates of Rosser and Schoenfeld \cite{RS62},
$\theta(x)<1.01624\,x$ for all $x>0$ and
$\theta(x)>x\bigl(1-\frac1{\log x}\bigr)$ for $x\ge41$, whence in
particular
\begin{equation}\label{eq:theta084}
\theta(x)\ \ge\ 0.84\,x\qquad(x\ge101):
\end{equation}
for $x\ge523$ this follows from $1-\frac1{\log x}>0.84$, and for
$101\le x<523$ it is checked directly from the list of primes below
$523$.

Let $n\ge4$ and consider a tight $[n-1]_{r\mapsto mr}$ of Goddyn--Wong
type. If $ms-1\ge101$, then \eqref{eq:theta} and the two estimates give
\[
0.84\,(ms-1)\ <\ \log n+1.01624\,(s-1),
\qquad\text{hence}\qquad
ms\ <\ 1.1905\log n+1.2099\,(s-1)+1 .
\]
Dividing by $s\ge2$ we obtain
\[
m\ <\ 1.2099+\frac{1.1905\log n-0.2099}{s}
\ \le\ 0.5953\log n+1.105 ,
\]
and therefore $w=mr<mn<0.60\,n\log n+1.11\,n$. If $ms-1<101$, then
$m\le\frac{101}{s}\le50.5$ and $w<mn\le50.5\,n$. Together with the
sporadic speeds $w\le9$ this gives
$M(n)<\max(0.60\,n\log n+1.11\,n,\ 50.5\,n,\ 9)<0.60\,n\log n+52\,n$
for $n\ge4$, and the bound is trivial for $n=3$.
\end{proof}

\begin{proof}[Proof that $\limsup M(n)/(n\log n)\le\frac12$]
Fix $\varepsilon\in(0,\frac12)$. By the prime number theorem there is
$x_\varepsilon$ with $\theta(x)\ge(1-\varepsilon)x$ for
$x\ge x_\varepsilon$. Consider a tight $[n-1]_{r\mapsto mr}$ as above.
If $ms-1<x_\varepsilon$, then $w<mn\le x_\varepsilon n$. Otherwise
\eqref{eq:theta} gives $(1-\varepsilon)(ms-1)<\log n+1.02\,s$, so that,
using $s\ge2$,
\[
m\ <\ \frac{\log n}{2(1-\varepsilon)}+\frac{1.02}{1-\varepsilon}+\frac12
\ <\ \frac{\log n}{2(1-\varepsilon)}+3 .
\]
Hence
$M(n)\le\frac{n\log n}{2(1-\varepsilon)}+\max(3,x_\varepsilon,9)\,n$ for
$n\ge4$, and letting $\varepsilon\to0$ proves the claim.
\end{proof}

\begin{proof}[Proof of the lower bound]
Let $p\ge3$ be prime, $q$ the next prime, $r=p\#$, $n=r+2$ (so $s=2$),
and $m=\frac{q-1}2\ge2$. Every integer $b\in[2,2m-1]=[2,q-2]$ has a
prime factor at most $q-2$, hence at most $p$, which divides $r$. Thus
\eqref{eq:gwcond} holds, and $[n-1]_{r\mapsto mr}$ is tight by
Theorem~\ref{thm:GW}, with
\[
w=mr=\frac{q-1}2\,(n-2).
\]
Since $\log n=\theta(p)+o(1)$ and $\theta(p)\sim p\sim q$ by the prime
number theorem, $w=(\frac12+o(1))\,n\log n$. The multiplier $m$ is
optimal for this pair $(r,s)$: for $m+1$ the interval $[2,2m+1]$
contains the prime $q\nmid r$.
\end{proof}

\begin{remark}
Along the primorials the ratio $w/(n\log n)$ approaches $\frac12$ slowly
and not monotonically, since it depends on the gap $q-p$. Its largest
value is $0.9246$, at $p=7$ ($n=212$, $r=210$, $w=1050$), and it still
exceeds $0.67$ at $p=37$, where $n\approx7.4\cdot10^{12}$.
The term $52\,n$ in Theorem~\ref{thm:B} dominates $0.60\,n\log n$
for all $n$ below about $4\cdot10^{37}$; at computable sizes the
effective content of the upper bound is a linear bound, and the
constant $\frac12$ is a genuinely asymptotic statement.
\end{remark}

%======================================================================
\section{Several accelerated runners}\label{sec:multi}
%======================================================================

The sufficiency part of Theorem~\ref{thm:GW} is a special case of a more
general result of Goddyn and Wong, in which several runners are
accelerated simultaneously.

\begin{theorem}[Goddyn--Wong {\cite[Thm.~3.1]{GoddynWong2006}}]\label{thm:gwmulti}
Let $m_1,\dots,m_{n-1}$ be positive integers and
$[n-1]_{\mathbf m}:=\{m_1,2m_2,\dots,(n-1)m_{n-1}\}$. If
\begin{equation}\label{eq:gwmulti}
\gcd(r,b)>1\quad\text{for every } r\in[n-1] \text{ with } m_r\ge2
\text{ and every } b\in[\,n-r,\ m_r(n-r)-1\,],
\end{equation}
then $\LR([n-1]_{\mathbf m})=\frac1n$.
\end{theorem}

Goddyn and Wong illustrate Theorem~\ref{thm:gwmulti} by the tight set
$([73]\setminus\{70,72\})\cup\{140,144\}$, $n=74$, which is also
discussed by Tao \cite{Tao2018}, and they construct tight sets in which
arbitrarily many runners are accelerated \cite[Sec.~3]{GoddynWong2006}.
The following simple pattern with all multipliers equal to $2$ makes
the case of up to three accelerated runners explicit.

\begin{example}[Doublings]\label{ex:doubling}
For $1\le k\le3$ let $s_i=2^i$ and let $P_i$ be the product of the
primes in $[s_i,2s_i-1]$, so $P_1=6$, $P_2=35$, $P_3=143$. If
\[
n\equiv s_i\pmod{P_i}\qquad(1\le i\le k),
\]
then
$V=\bigl([n-1]\setminus\{n-s_i\}_{i\le k}\bigr)\cup\{2(n-s_i)\}_{i\le k}$
is tight. By the Chinese Remainder Theorem the admissible $n$ form the
residue classes $n\equiv2\pmod6$ ($k=1$; take $n\ge8$),
$n\equiv74\pmod{210}$ ($k=2$) and $n\equiv23174\pmod{30030}$ ($k=3$).

Indeed, put $r_i=n-s_i$ and let $b\in[s_i,2s_i-1]$. If $b$ is prime,
then $b\mid P_i\mid r_i$. If $b$ is composite, then, as $i\le3$, $b$ is
divisible by $2$ or by $3$. All $r_i$ are even because
$n\equiv2\pmod6$. The only odd composites that occur are $9$ and $15$
(for $i=3$), and $3\mid r_3$ since $r_3=n-8\equiv0\pmod3$. Hence
\eqref{eq:gwmulti} holds. The pattern breaks down at $k=4$: the
conditions force $r_4\equiv1\pmod3$ and $r_4\equiv2\pmod7$, so
$\gcd(r_4,21)=1$ with $21\in[16,31]$.
\end{example}

The first instances with $k=2$ are $n=74$, the example of Goddyn and
Wong, and $n=284$. Both were also verified independently in exact
arithmetic (Section~\ref{sec:comp}).

The divisibility argument of Lemma~\ref{lem:div} has the following
variant, valid for arbitrary sets.

\begin{lemma}\label{lem:multiples}
If $V\subset\N$ satisfies $\LR(V)\le\frac1n$, then for every
$r\in[n-1]$ the set $V$ contains a multiple of $r$.
\end{lemma}

\begin{proof}
At $t=\frac1r$ some $v\in V$ has $\norm{v/r}\le\frac1n<\frac1r$, which
forces $r\mid v$.
\end{proof}

\begin{corollary}[Growth bound under the Goddyn--Wong criterion]\label{thm:growth}
Let $V=([n-1]\setminus R)\cup W$ be tight, and suppose that every $w\in
W$ has the form $w=mr$ with $r\in R$, $m\ge2$, and \eqref{eq:gwcond}
for the pair $(r,m)$. Then $\max V<0.60\,n\log n+52\,n$. This applies to
every set produced by Theorem~\ref{thm:gwmulti}.
\end{corollary}

\begin{proof}
The proof of the upper bound in Theorem~\ref{thm:B} used only
\eqref{eq:gwcond} for the pair $(r,m)$, and applies to each $w\in W$
separately.
\end{proof}

%======================================================================
\section{Two-entry modifications}\label{sec:twoswap}
%======================================================================

We now consider $V=([n-1]\setminus R)\cup W$ with $R=\{r_1,r_2\}$ and
$W=\{w_1,w_2\}$, $n\le w_1<w_2$. By Lemma~\ref{lem:cov}, if $V$ is tight
then
\[
U(n,R)\ \subseteq\ B_{w_1}\cup B_{w_2}.
\]
This is a \emph{union} condition. Different parts of one component may
be covered by different speeds, so the window argument of
Section~\ref{sec:thmA} does not carry over directly. A counting
argument does, however, as long as one removed speed is small.

\subsection{Two tools}

\begin{lemma}[Covering density]\label{lem:density}
Let $w\ge n$. For every interval $J\subseteq[0,1)$,
\[
\lvert B_w\cap J\rvert\ \le\ \Bigl(w\lvert J\rvert+\tfrac2n+1\Bigr)\frac2{wn}.
\]
\end{lemma}

\begin{proof}
A bad interval meeting $J$ has its centre $\frac jw$ in the
$\frac1{wn}$-neighbourhood of $J$. So $j$ ranges over an interval of
length $w\lvert J\rvert+\frac2n$, which contains at most
$w\lvert J\rvert+\frac2n+1$ integers, and each bad interval has length
$\frac2{wn}$.
\end{proof}

\begin{lemma}[Edge reach]\label{lem:edge}
Let $q\ge2$ and $\mu\ge2$ be integers with $\mu q<n\le w$, and put
$a=\frac1q+\frac1{\mu qn}$. If some bad interval of $w$ contains $a$,
then $w\ge\mu(n-q)$.
\end{lemma}

\begin{proof}
Let $c=\frac jw$ be the centre, so $\lvert c-a\rvert\le\frac1{wn}$.
Since $w>\mu q$ we have $\frac1{wn}<\frac1{\mu qn}$, hence
$c>\frac1q$, so $jq-w$ is a positive integer and
$\frac1{wq}\le c-\frac1q\le\frac1{\mu qn}+\frac1{wn}$. Multiplying by
$wqn$ gives $n\le\frac w\mu+q$.
\end{proof}

\subsection{\texorpdfstring{Removing the speed $2$}{Removing the speed 2}}

\begin{theorem}\label{thm:r2large}
Let $r'\in[n-1]\setminus\{2\}$ be arbitrary. If $n\ge19$ is odd or
$n\ge20$ is even, then no set $([n-1]\setminus\{2,r'\})\cup\{w_1,w_2\}$
with $w_1,w_2\ge n$ is tight.
\end{theorem}

\begin{proof}
Suppose such a set is tight. By \eqref{eq:monotone} and
Corollary~\ref{cor:U2}, $L_-\cup L_+=U(n,2)\subseteq U(n,\{2,r'\})$, so
$B_{w_1}\cup B_{w_2}\supseteq L_-\cup L_+$. As $B_{w_1}\cup B_{w_2}$ is
closed, it contains $\inf L_+=\frac12+\frac1{4n}$, and
Lemma~\ref{lem:edge} (with $q=\mu=2$) shows that one of the speeds, say
$w_*$, satisfies $w_*\ge2(n-2)$; the other is at least $n$. Let
$\ell=\Delta(n)$ be the common length of $L_\pm$. By
Lemma~\ref{lem:density}, the part of $B_w$ in each component has measure
at most $\beta_w=(w\ell+\frac2n+1)\frac2{wn}$, which decreases in $w$.
Hence the measure that can be covered is at most
\[
\Phi(n)=2\bigl(\beta_{2(n-2)}+\beta_n\bigr),
\]
and coverage is impossible if $\Phi(n)<2\ell$. For odd $n$
($\ell=\frac1{4n}$) one finds
$\Phi(n)=\frac{3n+2}{n^2(n-2)}+\frac{5n+8}{n^3}$, and $\Phi(n)<2\ell$
is equivalent to $n^3-18n^2+32>0$, which holds for $n\ge19$. For even
$n$ ($\ell=\frac{n-3}{4n(n-1)}$) the inequality $\Phi(n)<2\ell$ is
equivalent to
\[
n^4-21n^3+30n^2+16n-32>0,
\]
which holds for $n\ge20$: the left side is $4288$ at $n=20$ and
increasing for $n\ge15$.
\end{proof}

\begin{theorem}[$R=\{2,3\}$]\label{thm:r23}
Let $n\ge8$ and $w_1,w_2\ge n$. If $n\ge19$, the set
$([n-1]\setminus\{2,3\})\cup\{w_1,w_2\}$ is never tight. If
$8\le n\le18$ and $w_1,w_2\le40n$, it is tight if and only if $n=8$ and
$\{w_1,w_2\}=\{11,13\}$, i.e.\ $V=\{1,4,5,6,7,11,13\}$.
\end{theorem}

\begin{proof}
The case $n\ge19$ is Theorem~\ref{thm:r2large}. The tightness of
$\{1,4,5,6,7,11,13\}$ is classical \cite{GoddynWong2006} and was
confirmed in exact arithmetic. For $8\le n\le18$ the statement is a
finite computation. All pairs $n\le w_1<w_2\le40n$ surviving the
edge-reach conditions are tested in exact rational arithmetic against
the exact components of $U(n,\{2,3\})$ (Section~\ref{sec:comp}).
\end{proof}

\subsection{Arbitrary small removed speeds}\label{sec:qgeneral}

We now run the argument of Theorem~\ref{thm:r2large} at $\frac1q$
instead of $\frac12$. Throughout this subsection
\[
q\ge2,\qquad 4q\le n-1,\qquad R=\{q,r'\}\ \text{with}\
r'\in[n-1]\setminus\{q\},\qquad S=[n-1]\setminus R,
\]
and
\[
F_{\pm1}=\max\{i\in S:\ i\equiv\pm1\pmod q\},\qquad
m_0q=\min\{i\in S:\ q\mid i\}.
\]
Each residue class has a representative in $\{n-q,\dots,n-1\}$ and one
in $\{n-2q,\dots,n-q-1\}$, and at most one of them is removed; hence
$n-2q\le F_{\pm1}\le n-1$. Likewise $m_0\in\{2,3\}$, since at most one
of $2q,3q\le n-1$ is removed. From $4q\le n-1$ we get
\begin{equation}\label{eq:qcrude}
F_{\pm1}\ \ge\ n-2q\ \ge\ \frac{n+1}2,\qquad
\beta_\pm:=\frac{1/q-1/n}{F_{\mp1}}\ \le\ \frac{n-q}{qn(n-2q)}
\ <\ \frac2{q(n+1)}\ <\ \frac1n .
\end{equation}

\begin{lemma}[Windows at $1/q$]\label{lem:qwin}
The intervals
\[
W_-=\Bigl(\tfrac1q-\beta_-,\ \tfrac1q-\tfrac1{m_0qn}\Bigr),\qquad
W_+=\Bigl(\tfrac1q+\tfrac1{m_0qn},\ \tfrac1q+\beta_+\Bigr)
\]
are nonempty and contained in $U(n,R)$.
\end{lemma}

\begin{proof}
We treat $W_+$; for $W_-$ replace $\tau$ by $-\tau$ and interchange the
classes $1$ and $-1$. Write $t=\frac1q+\tau$ with
$\frac1{m_0qn}<\tau<\beta_+$.

\emph{Nonemptiness.} The inequality $\frac1{m_0qn}<\beta_+$ is
equivalent to $F_{-1}<m_0(n-q)$, which holds since
$F_{-1}\le n-1<2(n-q)$.

\emph{Multiples of $q$.} Let $kq\in S$, so $k\ge m_0$. The constraint
$\norm{kq\tau}>\frac1n$ fails only for $\tau$ within $\frac1{kqn}$ of
some $\frac j{kq}$. For $j=0$ this region lies below
$\frac1{m_0qn}$. For $j\ge1$ it lies above
$\frac{1-1/n}{kq}\ge\frac1n>\beta_+$, because $kq\le n-1$.

\emph{Other classes.} Let $i\in S$ with $\rho=i\bmod q\in[1,q-1]$; then
$it\equiv\frac\rho q+i\tau\pmod1$. We claim
$i\tau<\frac{q-\rho}q-\frac1n$. For $\rho=q-1$ this follows from
$i\le F_{-1}$. For $\rho\le q-2$ we use $i\le n-1<2F_{-1}$ from
\eqref{eq:qcrude}:
\[
i\tau<\frac{i(n-q)}{qnF_{-1}}<\frac{2(n-q)}{qn}<\frac{2n-q}{qn}
\le\frac{(q-\rho)n-q}{qn}=\frac{q-\rho}q-\frac1n .
\]
Thus $z=\frac\rho q+i\tau\in(\frac\rho q,1-\frac1n)$, and
$\norm{z}>\frac1n$: if $z\le\frac12$, then $\norm z\ge\frac\rho
q\ge\frac1q>\frac1n$, and otherwise $\norm z=1-z>\frac1n$.
\end{proof}

\begin{theorem}[A small removed speed]\label{thm:qsmall}
Let $q\ge2$ and $n\ge\max(40,10q)$, let $r'\in[n-1]\setminus\{q\}$ be
arbitrary, and let $w_1,w_2\ge n$. Then
$([n-1]\setminus\{q,r'\})\cup\{w_1,w_2\}$ is not tight.
\end{theorem}

\begin{proof}
The hypotheses imply $4q\le n-1$. Suppose that the set is tight. Then
$B_{w_1}\cup B_{w_2}$ covers $W_-\cup W_+$ by Lemmas~\ref{lem:cov}
and~\ref{lem:qwin}, and, being closed, it contains the inner endpoint
$\frac1q+\frac1{m_0qn}$ of $W_+$. By Lemma~\ref{lem:edge} with
$\mu=m_0$, one speed $w^*$ satisfies $w^*\ge m_0(n-q)\ge2(n-q)$; the
other speed $w'$ is at least $n$. Let
$\Lambda=\lvert W_-\rvert+\lvert W_+\rvert$. Using $F_{\pm1}\le n-1$ and
$m_0\ge2$,
\[
\Lambda\ \ge\ 2\ell^*,\qquad
\ell^*=\frac{n-q}{qn(n-1)}-\frac1{2qn}=\frac{n-2q+1}{2qn(n-1)} .
\]
Applying Lemma~\ref{lem:density} to both speeds and both windows,
\[
\Lambda\ \le\ \frac{4\Lambda}n+\frac{4(1+2/n)}n
\Bigl(\frac1{2(n-q)}+\frac1n\Bigr),
\]
so coverage forces
\begin{equation}\label{eq:qdens}
\ell^*\Bigl(1-\frac4n\Bigr)\ \le\
\frac{1+2/n}{n(n-q)}+\frac{2(1+2/n)}{n^2}.
\end{equation}
We show that \eqref{eq:qdens} fails. Using $n-2q+1\ge\frac{4n}5$,
$\frac{n-4}n\ge\frac9{10}$ and $n-1<n$,
\[
\ell^*\Bigl(1-\frac4n\Bigr)=\frac{(n-2q+1)(n-4)}{2qn^2(n-1)}
\ \ge\ \frac9{25\,qn}.
\]
Using $1+\frac2n\le\frac{21}{20}$ and $n-q\ge\frac{9n}{10}$, the right
side of \eqref{eq:qdens} is at most
$\frac{21}{20}\bigl(\frac{10}9+2\bigr)\frac1{n^2}=\frac{49}{15n^2}$.
Finally $\frac9{25qn}>\frac{49}{15n^2}$ is equivalent to $135n>1225q$,
which follows from $q\le\frac n{10}$.
\end{proof}

\begin{remark}[Why density beats windows here]\label{rem:density}
Each speed's bad set has density exactly $2/n$ on large scales ($w$
intervals of length $\frac2{wn}$), so two speeds cover at most a
$\tfrac4n$-fraction of any region, up to the per-component slack of
Lemma~\ref{lem:density}. Since $|L_-\cup L_+|$ is a fixed
$\approx\frac1{2n}$ and no constraint of the second removed speed can
shrink it ($U(n,2)\subseteq U(n,R)$ by \eqref{eq:monotone}, as used in
the proof of Theorem~\ref{thm:r2large}), coverage is impossible once
$\tfrac4n$ plus slack drops below $1$. The argument is insensitive to
the split phenomenon---it does not care \emph{which} speed covers
\emph{which} part, within or across components---and requires no upper
bound on the speeds. Run at $\tfrac1q$ rather than $\tfrac12$, the
mechanism needs a guaranteed window of measure $\gg\frac1{n^2}$
adjacent to $\tfrac1q$, which the class structure modulo $q$ supplies as
long as $q\le n/10$. What it cannot handle is two removals both close
to $\tfrac n2$: the guaranteed region $U(n,r_1)\cup U(n,r_2)$ can have
measure as small as $O\bigl(\frac{n-2r}{n^2}\bigr)$, which tends to
$0$ as $r\to\frac{n-1}2$, while the slack accumulates over its
$\asymp\varphi(r)$ components. That regime is settled for $8\le n\le140$,
$w\le8n$ by the census (Theorem~\ref{thm:twoswap}) and remains open in
general (Question~\ref{q:twoswap}).
\end{remark}

\begin{remark}[Sharper thresholds]\label{rem:qrefine}
After clearing denominators, the failure of \eqref{eq:qdens} is
equivalent to
\[
G(n,q)=n(n-q)(n-4)(n-2q+1)-2q(n-1)(n+2)(3n-2q)>0,
\]
whose leading part is $n^4(6c^2-9c+1)$ with $c=q/n$. The method
therefore applies to all $q\le(\gamma-\varepsilon)n$ for
$n\ge n_1(\varepsilon)$, where $\gamma=\frac{9-\sqrt{57}}{12}=0.1208\ldots$
is the smaller root of $6c^2-9c+1$. The exact thresholds beyond which
$G(n,q)>0$ holds for all larger $n$ are $n_0(2)=20$, $n_0(3)=28$,
$n_0(4)=36$, $n_0(5)=45$, $n_0(10)=86$, $n_0(20)=169$ and $n_0(50)=417$.
\end{remark}

\begin{corollary}\label{cor:qcombine}
Let $n\ge40$. Every tight set $([n-1]\setminus\{r_1,r_2\})\cup
\{w_1,w_2\}$ with $w_1,w_2\ge n$ satisfies $\min(r_1,r_2)=1$ or
$\min(r_1,r_2)>n/10$.
\end{corollary}

\subsection{The census}\label{sec:census}

\begin{theorem}[Census]\label{thm:twoswap}
Let $8\le n\le140$, $1\le r_1<r_2\le n-1$ and $n\le w_1<w_2\le8n$. The
set $([n-1]\setminus\{r_1,r_2\})\cup\{w_1,w_2\}$ is tight only in the
following two cases:
\begin{itemize}
\item $n=8$, $R=\{2,3\}$, $W=\{11,13\}$, giving $\{1,4,5,6,7,11,13\}$;
\item $n=74$, $R=\{70,72\}$, $W=\{140,144\}$, the example of Goddyn and
Wong (Example~\ref{ex:doubling} with $k=2$).
\end{itemize}
\end{theorem}

The theorem is the result of an exact computation, described in
Section~\ref{sec:comp}. The bound $w\le8n$ is a pragmatic window: it
contains the doubled speeds $2(n-2)$ of Example~\ref{ex:doubling} with
a fourfold margin, but no general bound on the inserted speeds of tight
two-entry modifications is known (Question~\ref{q:growth}).
Theorems~\ref{thm:r2large} and~\ref{thm:qsmall} take over where the
census stops, provided one removed speed lies in $[2,n/10]$; by
Corollary~\ref{cor:qcombine}, the remaining open case for $n\ge40$ is
that both removed speeds exceed $n/10$ or one of them equals $1$.

%======================================================================
\section{Computations}\label{sec:comp}
%======================================================================

All computations use exact rational arithmetic. Tightness of $V$, i.e.\
$\LR(V)=\frac1n$, is equivalent to the conjunction of two finite
conditions: (i) the bands $\{t:\norm{vt}\le\frac1n\}$, $v\in V$, cover
the circle, which is decided by the finitely many endpoints
$\frac jv\pm\frac1{vn}$; and (ii) some $t$ has $\norm{vt}\ge\frac1n$
for all $v\in V$, which is decided by evaluating
$f_V(t)=\min_v\norm{vt}$ at its finitely many breakpoints (the peaks
$\frac{2k+1}{2v}$ and the crossings $\frac k{v\pm u}$). For two-entry
modifications we use three independent implementations:
\begin{enumerate}
\item an \emph{edge-reach census}: for each pair $R$, the inserted
speeds are restricted by the necessary conditions of
Lemma~\ref{lem:edge} type (a bad interval covering the left end $a$ of
a component has its centre index in $[wa-\frac1n,wa+\frac1n]$). The
surviving pairs are screened in floating point, with bad intervals
enlarged so that only non-covering pairs can be discarded, and every
survivor is re-verified exactly;
\item a \emph{union filter}, intersecting the bad intervals with the
exact components of $U(n,R)$;
\item a \emph{remaining-region method}, computing the set where both
$\norm{w_1t}>\frac1n$ and $\norm{w_2t}>\frac1n$ and checking that it
avoids $U(n,R)$. This method does not use the component structure.
\end{enumerate}
They agree on all overlapping ranges. Table~\ref{tab:comp} lists the
verifications. Sizes count candidate cases after the (necessary)
edge-reach masks.

\begin{table}[t]
\centering
\small
\begin{tabular}{@{}>{\raggedright\arraybackslash}p{0.33\textwidth}
>{\raggedright\arraybackslash}p{0.25\textwidth}r
>{\raggedright\arraybackslash}p{0.22\textwidth}@{}}
\toprule
verification & range & size & outcome\\
\midrule
\multicolumn{4}{@{}l}{\emph{One-entry modifications (independent check of Theorem~\ref{thm:A})}}\\
middle range, effective bound & $r\le30$, $n\le6r$, $w\le\frac{4rI}{2s-I}$ & 21{,}789 & tight only $(5,2,7)$, $(6,2,9)$\\
brute force, no effective bound & $3\le r\le8$, $w\le30n$ & --- & none tight\\
\midrule
\multicolumn{4}{@{}l}{\emph{Two-entry modifications}}\\
census, all pairs $R$ (Thm.~\ref{thm:twoswap}) & $8\le n\le140$, $w\le8n$ & 447{,}545 & the two sets of Thm.~\ref{thm:twoswap}\\
$R=\{2,r'\}$, deep window & $9\le n\le18$, $w\le40n$ & 96{,}725 & none tight\\
$R=\{2,3\}$, deep window & $n=8$, $64<w_2\le320$ & 901 & none covers\\
\midrule
\multicolumn{4}{@{}l}{\emph{Three- and four-entry modifications (Question~\ref{q:kgw})}}\\
$k=3$, doubling type & $3\le r\le35$, $n\le4r$ & 241{,}699 & none tight\\
$k=4$, all multipliers & $10\le n\le30$ & 12{,}646 & none tight\\
\midrule
\multicolumn{4}{@{}l}{\emph{Enumeration subtotal: $821{,}305$ cases, $0$ unexpected tight sets.}}\\
\midrule
\multicolumn{4}{@{}l}{\emph{Further exact checks}}\\
$M(n)$ via \eqref{eq:gwcond} & $n\le40{,}000$ & --- & $<0.60\,n\log n+52n$\\
primorial family & $p\le19$ & --- & $m=\frac{q-1}2$ optimal\\
inequality \eqref{eq:qdens} & $q\le100$, $n\le5000$ & --- & holds for $n\ge\max(40,10q)$\\
deletion values & $n\le30$, all $r$ & 432 & consistent with Sec.~\ref{sec:prelim}\\
spectrum (Table~\ref{tab:spec}) & $5\le n\le12$, $r=2$, $w\le8n$ & 484 & smallest non-tight value as listed\\
Example~\ref{ex:doubling}, $n=74$, $284$ & --- & 2 & tight\\
\bottomrule
\end{tabular}
\caption{Exact-arithmetic verifications.}
\label{tab:comp}
\end{table}

\begin{remark}[Audit of the theoretical inequalities]\label{rem:audit}
Beyond the searches, each inequality used in
Section~\ref{sec:thmA} was tested directly as an independent
exact-arithmetic check, although Lemma~\ref{lem:key} covers all $r\ge3$
unconditionally: the equivalence of Lemma~\ref{lem:win} with the exact
integer-existence condition, and the sign claims of
Lemma~\ref{lem:sign}, over all $2{,}628$ admissible triples
$(n,r,w)$ with $3\le r\le14$ and $w\le4rI/(2s-I)$, $2r\le n-1$ (no
discrepancy, and no triple satisfied both windows); the implication
$\omega(r)\ge2\Rightarrow2s-I\ge5$ and the threshold chain
$2s-I\ge5\Rightarrow4rI/(2s-I)<r(n-2)$, as well as the key inequality
itself, for all $r\le60$ and all admissible $n$ (no violation). The
estimate \eqref{eq:theta084} was checked for all integers
$101\le x\le2\cdot10^6$ in high-precision arithmetic against the exact
primorials: the minimum of
$\theta(x)/x$ over integer $x\ge101$ on the range is $0.84976$, attained
at $x=126$ (over real $x$ it is $\theta(113)/127=0.84307$), both above
$0.84$. The thresholds of Lemma~\ref{lem:unif} were recomputed exactly
and cross-checked against the census data.
\end{remark}

\subsection{\texorpdfstring{Loneliness values just above $1/n$}{Loneliness values just above 1/n}}\label{sec:spectrum}

By \eqref{eq:sandwich} and Theorem~\ref{thm:r2closed}, every value
$\LR([n-1]_{2\to w})$ lies in $[\frac1n,\frac2D]$, an interval that
contains $(\frac1n,\frac1{n-1})$. By the theorem of Giri and Kravitz
\cite{GiriKravitz2026}, the loneliness values of $(n-2)$-element sets,
in particular $\frac1{n-1}=\LR([n-2])$, are accumulation points of the
values of $(n-1)$-element sets. Table~\ref{tab:spec} lists, for
$5\le n\le12$, the smallest value of $\LR([n-1]_{2\to w})$ exceeding
$\frac1n$ over the searched range $n\le w\le W_0=8n$ (the census window
of Section~\ref{sec:comp}); each value is certified in exact rational
arithmetic. The following lemma, applied at the table values, promotes
the search to a complete classification over all $w\ge n$.

\begin{lemma}[A uniform covering bound for $r=2$]\label{lem:unif}
Let $n\ge5$, let $i_0$ be the largest odd integer $\le n-1$ (so that
$D=2\lfloor n/2\rfloor+3=i_0+4$), and let $0<\lambda<\frac2D$. Put
\[
\ell(\lambda)=\frac{1/2-\lambda}{i_0}-\frac\lambda4\quad(>0).
\]
If $w\ge n$ and $\LR([n-1]_{2\to w})\le\lambda$, then
$w\le2\lambda/\ell(\lambda)$.
\end{lemma}

\begin{proof}
Let $J=\bigl(\frac12+\frac\lambda4,\ \frac12+\frac{1/2-\lambda}{i_0}\bigr)$,
which is nonempty, and write $t=\frac12+\tau\in J$. For odd $i\le i_0$
we have $it\equiv\frac12+i\tau\pmod1$ with $0<i\tau<\frac12-\lambda$, so
$\norm{it}=\frac12-i\tau>\lambda$. For even $i$ with $4\le i\le n-1$ we
have $it\equiv i\tau\pmod1$ and
\[
\lambda<i\tau<\frac{(n-1)(1/2-\lambda)}{i_0}
\le\frac43\Bigl(\frac12-\lambda\Bigr)<1-\lambda ,
\]
so again $\norm{it}>\lambda$. Hence every speed of $[n-1]_{2\to w}$
other than $w$ exceeds $\lambda$ on $J$, and $\LR\le\lambda$ forces
$\norm{wt}\le\lambda$ on $\overline J$. As $\lambda<\frac12$, the set
$\{t:\norm{wt}\le\lambda\}$ is a disjoint union of closed intervals of
length $\frac{2\lambda}w$, one of which contains $\overline J$; thus
$\ell(\lambda)\le\frac{2\lambda}w$.
\end{proof}

For $\lambda=\frac1n$ one has $\ell(\frac1n)=\Delta(n)$, and the lemma
reduces to the length condition in the proof of
Proposition~\ref{prop:r2}. Since $2\lambda/\ell(\lambda)$ increases with
$\lambda$, the lemma shows that $\LR([n-1]_{2\to w})>v$ whenever
$w>2v/\ell(v)$. For the values $v=v_n$ of Table~\ref{tab:spec} one finds
$2v_n/\ell(v_n)=12,20,10,14,14,18,18,22$ for $n=5,\dots,12$, which is
well inside the searched range $w\le8n$. Consequently
Table~\ref{tab:spec} lists the smallest non-tight value of
$\LR([n-1]_{2\to w})$ over \emph{all} $w\ge n$; in particular, for
$9\le n\le12$ no one-entry modification removing $2$ has loneliness in
$(\frac1n,\frac1{n-1})$.

\begin{table}[t]
\centering
\small
\begin{tabular}{rrrrcc}
\toprule
$n$ & $1/n$ & smallest value & at $w$ & $<\frac1{n-1}$? &
of the form $\frac{\sigma}{(n-1)\sigma+1}$?\\
\midrule
5 & $1/5$ & $2/9$ & 5 & yes & $\sigma=2$\\
6 & $1/6$ & $2/11$ & 7 & yes & $\sigma=2$\\
7 & $1/7$ & $2/13$ & 9 & yes & $\sigma=2$\\
8 & $1/8$ & $2/15$ & 11 & yes & $\sigma=2$\\
9 & $1/9$ & $2/15$ & 11 & no & ---\\
10 & $1/10$ & $2/17$ & 13 & no & ---\\
11 & $1/11$ & $2/17$ & 13 & no & ---\\
12 & $1/12$ & $2/19$ & 15 & no & ---\\
\bottomrule
\end{tabular}
\caption{Smallest non-tight values of $\LR([n-1]_{2\to w})$: certified
in exact arithmetic over the searched range $n\le w\le8n$, and valid for
all $w\ge n$ by Lemma~\ref{lem:unif}. The equal values at $n=10$ and
$n=11$ come from the distinct sets $[9]_{2\to13}$ and $[10]_{2\to13}$.}
\label{tab:spec}
\end{table}

For $n\le8$ these values lie strictly between $\frac1n$ and
$\frac1{n-1}$ and have the form $\frac\sigma{(n-1)\sigma+1}$ with
$\sigma=2$, as in Kravitz's conjecture \cite{Kravitz2021}. For $9\le
n\le12$ they exceed $\frac1{n-1}$. We record this as an observation
only. Kravitz's conjecture is false in general \cite{FanSun2026}, and
our data concern a very special class of sets.

%======================================================================
\section{Open problems}\label{sec:open}
%======================================================================

\begin{question}\label{q:twoswap}
Classify the tight two-entry modifications
$([n-1]\setminus\{r_1,r_2\})\cup\{w_1,w_2\}$. By
Corollary~\ref{cor:qcombine} and Theorem~\ref{thm:twoswap}, the open
cases are $n>140$ or $w_2>8n$, with both $r_i>n/10$ or $\min r_i=1$.
When $r_1,r_2>\frac{n-1}2$, Lemma~\ref{lem:multiples} shows that each
inserted speed covering $\frac1{r_i}$ is a multiple of $r_i$. Is every
tight two-entry modification with $r_1,r_2>\frac{n-1}2$ obtained from
Theorem~\ref{thm:gwmulti}?
\end{question}

\begin{question}\label{q:growth}
Does a bound $\max V=O(n\log n)$, or any polynomial bound, hold for all
tight modifications $([n-1]\setminus R)\cup W$ with $\lvert
R\rvert=\lvert W\rvert\ge2$? Corollary~\ref{thm:growth} gives such a
bound under the Goddyn--Wong criterion, but for $\lvert R\rvert\ge2$
tightness alone does not obviously imply the criterion for each
inserted speed.
\end{question}

\begin{question}\label{q:kgw}
Are there tight multi-entry modifications of $[n-1]$ that do not arise
from Theorem~\ref{thm:gwmulti}, apart from sets such as
$\{1,4,5,6,7,11,13\}$? The converse of Theorem~\ref{thm:gwmulti}
already fails in general \cite[Sec.~3]{GoddynWong2006}, but the
searches of Table~\ref{tab:comp} found no such example.
\end{question}

\begin{question}\label{q:deletion}
Is $\LR([n-1]\setminus\{r\})>\frac1{n-1}$ for all $n\ge5$ and all
$r\le n-2$? By Theorem~\ref{thm:delwit} and Corollary~\ref{cor:delwit}
this holds unless no integer in $[n+r,2n-3]$ is coprime to $r$; in
particular it can fail only for $r>n-2-2\sqrt n$. The exact value is
known for $r\le\frac{n-1}2$ by Liu's theorem \cite[Thm.~18]{Liu2008}.
\end{question}

\begin{question}\label{q:shifted}
In the notation of Appendix~\ref{sec:shifted}: is
$\Gamma(n)<\frac1{n+1}$ for every $n\ge5$? What are the minimisers of
$\Gamma(n)$?
\end{question}

\appendix
\section{The shifted variant}\label{sec:shifted}

For a vector $\mathbf v=(v_1,\dots,v_n)$ of distinct positive integers
and starting positions $\mathbf s\in[0,1)^n$, put
\[
\gamma(\mathbf v;\mathbf s)=\max_{t}\min_{i}\norm{s_i+v_it},\qquad
\gamma^{\min}(\mathbf v)=\min_{\mathbf s}\gamma(\mathbf v;\mathbf s),
\qquad
\Gamma(n)=\min_{\lvert\mathbf v\rvert=n}\gamma^{\min}(\mathbf v).
\]
The \emph{shifted Lonely Runner Conjecture}
\cite[Conj.~1]{BeckHostenSchymura2019} asserts
$\gamma^{\min}(\mathbf v)\ge\frac1{n+1}$. It is closely related to the
zonotopal reformulation of Henze and Malikiosis \cite{HenzeMalikiosis2017}
and to the finite-checking results of Malikiosis, Santos and Schymura
\cite{MalikiosisSantosSchymura2025}. Alc\'antara, Criado and Santos
proved it for four speeds and found exactly three primitive tight
instances, two of which are tight in the classical sense
\cite{AlcantaraCriadoSantos2026}. Blanco, Criado and Santos disproved it
for $5\le n\le17$; for instance
$\gamma^{\min}(1,2,3,4,5)=\frac{15}{94}<\frac16$
\cite{BlancoCriadoSantos2026} (the case $n=6$ being
$\gamma^{\min}(1,\dots,6)=\frac9{67}<\frac17$).
Writing $\Gamma(n)=1/(n+1+E_n)$, Poliakova \cite{Poliakova2026} proved
$E_n\ge\lfloor n/287\rfloor$, so that $E_n$ grows linearly, and
$E_n\ge1$ for all $n\ge95$. Together with the computations of
\cite{BlancoCriadoSantos2026}, Question~\ref{q:shifted} (posed in
version~2 of this paper) is therefore open only for $18\le n\le94$.

\begin{lemma}\label{lem:shiftmono}
Deleting a speed cannot decrease $\gamma^{\min}$. In particular
$\Gamma(n)$ is nonincreasing in $n$.
\end{lemma}

\begin{proof}
Let $\mathbf v'$ arise from $\mathbf v$ by deleting $v_j$, and let
$\mathbf s'$ be starting positions for $\mathbf v'$. Extend them by an
arbitrary $s_j$. For each $t$ the minimum over fewer coordinates is
larger, so $\gamma(\mathbf v';\mathbf s')\ge\gamma(\mathbf v;\mathbf s)
\ge\gamma^{\min}(\mathbf v)$; now minimise over $\mathbf s'$. For the
second claim, append a speed to a minimiser of $\Gamma(n)$.
\end{proof}

The contrast with the classical case is sharp. Classically the baseline
is tight for every $n$, and Theorems~\ref{thm:GW} and~\ref{thm:gwmulti}
supply infinitely many further tight sets. In the shifted setting the
baseline fails from five speeds on, and no analogue of the
Goddyn--Wong families is known. Lemma~\ref{lem:shiftmono} alone does
not propagate failure from $n$ to $n+1$, because the target
$\frac1{n+1}$ decreases.

\section*{Data availability}
The programs and run logs for Section~\ref{sec:comp} are available at
\url{https://github.com/Qinrayn/lonely-runner-tight-II}, one standalone
script in exact rational arithmetic per verification.

\section*{Declaration on the use of generative AI}
The mathematical content of this paper originates from the author's
research; the main theorems and the strategy of their proofs were
conceived by the author. The AI assistant Kimi~K3 (Moonshot AI) was
used for drafting parts of the exposition and for language editing.
All AI-assisted material was checked in detail by the author, who
takes full responsibility for the content of the paper.

\end{document}